\documentclass[12pt]{article}
\title{Line Complexity Asymptotics of Polynomial Cellular Automata}
\author{Bertrand Stone} 
\usepackage{amsmath}
\usepackage{amssymb}
\usepackage{relsize}
\usepackage{graphicx}
\usepackage{mathrsfs}
\usepackage{mathptmx}
\usepackage{amsthm}
\usepackage{setspace}
\usepackage{hyperref}

\newcommand{\Z}{\mathbb Z}

\newcommand{\algA}{\mathscr A}

\DeclareMathOperator*{\Sum}{\mathlarger{\sum}}

\newcommand{\frL}{\left\langle\log_p\frac{k}{j}\right\rangle}
\newcommand{\flL}{\left\lfloor\log_p\frac{k}{j}\right\rfloor}
\newcommand{\rL}{\log_p\frac{k}{j}}

\newtheorem{theorem}{Theorem}

\newtheorem{cor}{Corollary}
\newtheorem{prop}{Proposition}

\newtheorem{remark}{Remark}

\begin{document}

%Title page
\maketitle

\Large
\begin{center}\textbf{Abstract}\end{center}
\normalsize

Cellular automata are discrete dynamical systems which consist of changing patterns of symbols on a grid. An automaton changes from its present state to the next according to a transition rule which determines the automaton's local behavior. Despite the simplicity of their definition, cellular automata have been applied in the simulation of complex phenomena as disparate as biological systems and universal computers. In this paper, we will consider cellular automata that arise from polynomial transition rules, where the symbols in the automaton are integers modulo some prime $p$. We are principally concerned with the asymptotic behavior of the \emph{line complexity sequence} $a_T(k)$, which counts, for each $k$, the number of coefficient strings of length $k$ that occur in the automaton. We begin with the modulo $2$ case. For a given polynomial $T(x) = c_0 + c_1x + \ldots + c_nx^n$ with $c_0,c_n\neq 0$, we construct \emph{odd} and \emph{even} parts of the polynomial from the strings $0c_1c_3c_5\cdots$ and $c_0c_2c_4\cdots$, respectively. We prove that for polynomials whose odd and even parts are relatively prime, $a_T(k)$ satisfies recursions of a specific form. We also consider powers of transition rules in the modulo $p$ case, introducing a notion of the \emph{order} of a recursion, distinct from the order of the transition rule. We show that the property of ``having a recursion of some order'' is preserved when the transition rule is raised to a positive integer power. We then derive functional relations for the generating functions associated to the line complexity sequence, using the recursions described above. Extending to a more general setting, we investigate the asymptotics of $a_T(k)$ by considering an abstract generating function $\phi(z)=\sum_{k=1}^\infty\alpha(k)z^k$ which satisfies a general functional equation relating $\phi(z)$ and $\phi(z^p)$ for some prime $p$. We show that there is a continuous, piecewise quadratic function $f$ on $[1/p, 1]$ for which $\lim_{k\to\infty}\left\lbrack\frac{\alpha(k)}{k^2} - f(p^{-\langle\log_p k\rangle})\right\rbrack = 0$, where $\langle y\rangle$ denotes the fractional part of $y$. We use this result to show that for positive integer sequences $s_k\to\infty$ with a parameter $x\in [1/p,1]$ and for which $\lim_{k\to\infty}\langle\log_p s_k(x)\rangle=\log_p\frac{1}{x}$, the ratio $\alpha(s_k(x))/s_k(x)^2$ tends to $f(x)$, and that the limit superior and inferior of $\alpha(k)/k^2$ are given by the extremal values of $f$.

\newpage
\section{Introduction}

A \textit{cellular automaton} is a discrete system which consists of patterns of symbols on a grid. These patterns change in successive time intervals, and the changes are specified by a \emph{transition rule}, in such a way that the symbol in a particular location at a particular point in time is determined by the surrounding symbols in the previous state. Although cellular automata are determined by simple local rules, they can nevertheless exhibit large-scale complex behavior. Von Neumann, who initiated the study of cellular automata, investigated their connections to the modelling of biological systems \cite{vnm}. As Willson notes in \cite{Willson_fract}, cellular automata can be used to model chaotic phenomena because their discrete structure facilitates exact computation.

In this paper, we shall focus on one-dimensional cellular automata. A particular state for such an automaton is called a \emph{configuration}, and may be expressed as a Laurent series \[\Sum_{-\infty}^\infty{a_ix^i},\] where the superscripts correspond to the locations of the values $a_i$. For example, the expression $x + 3x^3 + 2x^4$ represents the string $01032$.

Given a configuration $\omega$, the \emph{transition rule} $T$ for a cellular automaton determines a new configuration $T\omega$ in such a way that the value at a given index $i$ in $T\omega$ is determined by values near $i$ in $\omega$. An \emph{additive} transition rule is specified by a Laurent polynomial and acts upon a configuration by multiplication. In this paper, we will use as an alphabet the integers modulo some prime $p$. In this case the transition rule acts upon a configuration by multiplication, and the coefficients are reduced modulo $p$. We illustrate this process by constructing Pascal's triangle modulo $2$ in Figure \ref{Pascal}; we take $p=2$, $T(x) = 1+x$, and start with the initial state $\omega_0 = 1$.

\begin{figure}[h]
\centering
\includegraphics{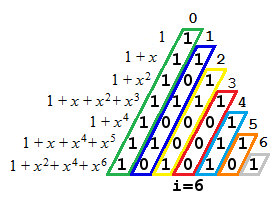}
\caption{Constructing Pascal's triangle modulo 2}
\label{Pascal}
\end{figure}

A more complicated example is obtained by taking $p=2$, $\omega_0 = 1$, $T(x) = 1+x^2+x^4+x^5$. This automaton is illustrated in Figure \ref{more}.

\begin{figure}
\centering
\includegraphics[scale = 0.45]{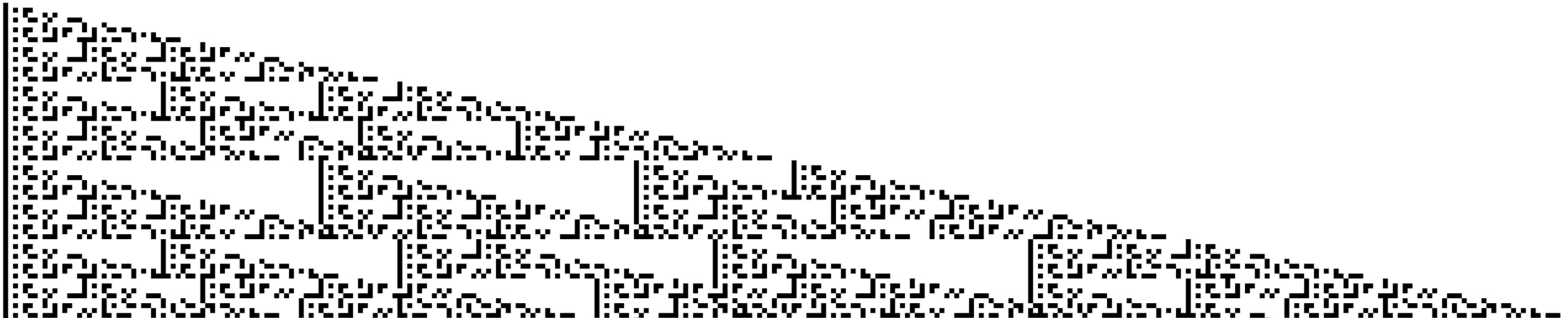}
\caption{The automaton obtained by iteratively multiplying $\omega_0 = 1$ by the rule $T(x) = 1+x^2+x^4+x^5$, modulo $2$.}
\label{more}
\end{figure}

Sequences of length $k$ which appear in some configuration are called $k$-\emph{accessible blocks}. For example, the block $110011$ appears in line $5$ of the automaton shown in Figure \ref{Pascal}, and is thus accessible. We will write $a_T(k)$ for the number of accessible blocks of length $k$ for a given transition rule $T$ (it is implicitly assumed that the initial state has been specified). We define $a_T(0) = 1$: the empty string is always accessible. The sequence $a_T(k)$ for $k \geq 0$ is called the \emph{line complexity} of the automaton. Unless otherwise specified, we will assume that the automaton has initial state $1$.

For transition rules of positive degree which do not reduce to monomials, and for which the coefficients are taken modulo general primes $p$, Berth\'e showed in \cite{berthe} that the line complexity sequence satisfies \begin{equation}\label{berthe-estimate}C_1 \leq\frac{a_T(k)}{k^2}\leq C_2,\end{equation} for some fixed positive constants $C_1, C_2$.

Garbe \cite{Garbe} considered the transition rule $T(x) = 1+x$ with coefficients taken modulo $p$ and the rule $T(x) = 1+x+x^2$ with coefficients taken modulo small primes $p$, and investigated the asymptotic behavior of subsequences of the quotient $a_T(k)/k^2$. In particular, for the sequence $n(k) = \lfloor p^k/x\rfloor$, Garbe showed that the quotient $a_T(n(k))/n(k)^2$ converges to a function that is piecewise quadratic in $x$. 

We will consider more general polynomials, and investigate the asymptotic behavior of the associated automata using recursion formulas for the line complexity sequence. We begin with the modulo 2 case. For a given polynomial $T(x) = c_0 + c_1x + \ldots + c_nx^n$ with $c_0,c_n\neq 0$, we construct \emph{odd} and \emph{even} parts of the polynomial from the strings $0c_1c_3c_5\cdots$ and $c_0c_2c_4\cdots$, respectively. We will prove that polynomials for which the odd and even parts are relatively prime have recursion formulas which we can use to investigate the asymptotics of the line complexity sequence. We will generalize Garbe's results on asymptotics to the present context: in particular, we will show that there is a piecewise quadratic function $f$ on $[1/p, 1]$ for which $\lim_{k\to\infty}\left\lbrack\frac{a_T(k)}{k^2} - f(p^{-\langle\log_p k\rangle})\right\rbrack = 0$, where $\langle y\rangle$ denotes the fractional part of $y$. We then investigate positive integer sequences $s_k$ with a parameter $x\in [1/p,1]$, such that $s_k\to\infty$ and $\lim_{k\to\infty}\langle\log_p s_k(x)\rangle=\log_p\frac{1}{x}$, and show that the ratio $a_T(s_k(x))/s_k(x)^2$ tends to $f(x)$. We also show that the limit superior and inferior of $a_T(k)/k^2$ are given by the extremal values of $f$, thus proving a more precise version of the bound (\ref{berthe-estimate}) for large $k$.

In Section 2, we introduce some useful notation. In Section 3, we will describe the general structure of the recursion relations, and we will see the importance of the injectivity of several transformations that we will introduce. In Section 4, we investigate the injectivity of these maps, and provide a complete characterization of which polynomials induce injective maps on the whole space. In Section 5, we show that the asymptotic sizes of the intersections that arise in Section 3 are constant. In Section 6, we examine some interesting consequences of introducing a notion of the \emph{order} of a recursion, and characterize the behavior of the line complexity sequence when the transition rule is raised to a power. In Section 7 we examine generating functions for the line complexity sequence in the modulo $2$ case. In Section 8, we investigate the asymptotics of the quotient $a_T(k)/k^2$ in a general context. We conclude and describe some directions of future research in Section 9.

\section{Notation}

In this section, we introduce some notation which we will use throughout the rest of the paper.

In the following, we will write $A_p(I; T)$ for the automaton generated by iteratively multiplying $I$ by $T$ and reducing the coefficients modulo $p$. We will assume throughout that $I, T\in (\Z/p)[x]$. We will write $\algA(k)$ for the set of accessible blocks of length $k$ associated to such an automaton.

We shall write $1^201 = 1101$ etc. in block notation; to distinguish this notation from operations such as squaring, we shall write the latter with square brackets, e.g. 
\[\lbrack(111)^2\rbrack=(1+x+x^2)^2 = 10101,\] whereas \[(111)^2 = 111111.\]

If $b = b_0\cdots b_n$, we will write $b|_i^j = b_i\cdots b_j$. At the end of a block, we employ the notation $0_l$ to represent sufficiently many zeros to bring the total length of the block to $l+1$; for example $1010_5 = 101000$.

If $f$ and $g$ are polynomials, we will write $(f,g)=1$ to indicate that $f$ and $g$ have no nontrivial common factors.

\section{Recursion Formulas for the Line Complexity Sequence}\label{recursion_formulas}
Our study of the asymptotic properties of the line complexity sequence is based upon recursion formulas for $a_T(2k)$ and $a_T(2k+1)$. These recursions hold for sufficiently large $k$, and their structure is motivated by the following analysis. We shall focus primarily on the recursion for $a_T(2k)$.

Consider an automaton $A_2(1; T)$, where $T$ is a polynomial of degree $n$, and some even row of this automaton, say $2r$. We see that this line of the automaton is of the form \begin{equation}T^{2r}(1) = T^{2r} = (T^r)^2.\label{goose}\end{equation}

In view of the identity $T(s^2)\equiv T(s)^2\pmod 2$, squaring a polynomial has the effect of inserting zeros between the original coefficients; thus (in block notation) we have
\[\lbrack(x_0x_1x_2)^2\rbrack = x_00x_10x_2.\] Since line $2r$ of the automaton is a square, we see that all accessible blocks of length $2k$ appearing in this row must be of the form
\[x_00x_10\cdots x_{k-1}0\] or of the form \[0x_00x_1\cdots 0x_{k-1}.\] Moreover, by the identity (\ref{goose}), it follows that $x_0x_1\cdots x_{k-1}$ must be accessible.

We introduce the sets
\[A_1 = \{x_00x_10\cdots x_{k-1}0:x_0x_1\cdots x_{k-1}\in\algA(k)\}\] and
\[A_2 = \{0x_00x_1\cdots 0x_{k-1}:x_0x_1\cdots x_{k-1}\in\algA(k)\},\] and the maps $T_{A_i}:\algA(k)\to A_i$ defined by
\[T_{A_1} : x_0x_1\cdots x_{k-1} \mapsto x_00x_10\cdots x_{k-1}0\] and \[T_{A_2} : x_0x_1\cdots x_{k-1} \mapsto 0x_00x_1\cdots 0x_{k-1}.\] It is clear that the maps $T_{A_i}$ are bijective, so that $|A_1| = |A_2| = a_T(k)$.

The accessible blocks in odd-numbered rows have a more complex structure. We first assume that $n$ is even, and consider a row $2r+1$ of the automaton. We want to establish a correspondence between accessible blocks of length $2k$ in this row and accessible blocks of some smaller length in row $r$. We will use the locally-determined nature of the automaton and the fact that all blocks in row $2r+1$ arise by applying the transition rule to row $2r$.

To produce the accessible blocks of length $2k$ in row $2r+1$, we start with a given accessible block $b$ of length $k+\frac{n}{2}$ in line $r$. It follows that the block $0\lbrack b^2\rbrack 0$ is a block of length $2k+n+1$ which appears in row $2r$. (Moreover, as we saw in the case of the sets $A_i$, all such blocks are produced in this way.) We now apply the transition rule to this block, obtaining a block of length $2k+2n+1$ in row $2r+1$ (the right side of the block must be padded with zeros to ensure this). We now eliminate the $n$ entries on either side of the resulting block, obtaining a block of length $2k+1$. This is necessary because in the context of the entire automaton, the block $b$ does not determine these $n$ entries on either side. This leaves a block $t$ of length $2k+1$. Finally, define
\[T_{B_1}b = t_0\cdots t_{2k-1}\] and \[T_{B_2}b = t_1\cdots t_{2k}.\]
Briefly, we can write
\[T_{B_i}b = T\left(0\lbrack b^2\rbrack 0\right)0_{2k+2n}|_{n+i-1}^{n+2k+i-2}.\]

For example, consider the automaton $A_2(1, 1+x+x^2)$. We outline the above process in the following schematic:

\begin{alignat*}{2}
& b \quad && 1011\\
& 0\lbrack b^2\rbrack 0 \quad && 010001010\\
& T(0\lbrack b^2\rbrack 0)0_{2k+2n} \quad && 01\vert \lefteqn{\underbrace{\phantom{110110}}_{T_{B_1}b}} 1\!\overbrace{101101}^{T_{B_2}b}\vert 10
\end{alignat*}

In the above example, we note that if there had been a $1$ immediately to the left of the block $0\lbrack b^2\rbrack 0$, the two leftmost entries of $T(0\lbrack b^2\rbrack 0)0_{2k+2n}$ would be changed to $10$. We thus see that these two entries cannot be determined by $b$ alone; this is why $n$ entries must be deleted on either side of $T(0\lbrack b^2\rbrack 0)0_{2k+2n}$.

We now define \[B_1 = T_{B_1}\left(\algA(k+\tfrac{n}{2})\right)\] and \[B_2 = T_{B_2}\left(\algA(k+\tfrac{n}{2})\right).\] Thus the maps $T_{B_i}:\algA(k+\frac{n}{2}) \to B_i$ are clearly surjective.

The case of odd $n$ is similar, but with some modification. Namely, in this case, the map $T_{B_1}$ acts upon blocks in $\algA(k+\frac{n-1}{2})$, and the map $T_{B_2}$ acts upon blocks in $\algA(k+\frac{n+1}{2})$. The sets $B_i$ are defined in the same way.

The transformations $T_{A_i}'$, $T_{B_i}'$ for blocks of \emph{odd} length are defined in an analogous manner. We first define $T_{A_1}'$ on $\algA(k+1)$ by

\[T_{A_1}':x_0x_1\cdots x_k \mapsto x_00x_10\cdots x_{k-1}0x_k,\] and $T_{A_2}'$ on $\algA(k)$ by \[T_{A_2}':x_0x_1\cdots x_{k-1}\mapsto 0x_00x_1\cdots 0x_{k-1}0.\] 

If $n$ is even, we define $T_{B_1}'$ on $\algA(k+\frac{n}{2})$ by \[T_{B_1}':b\mapsto T\left(0\lbrack b^2\rbrack 0\right)0_{2k+2n}|_{n}^{n+2k},\] and $T_{B_2}'$ on $\algA(k+\frac{n}{2}+1)$ by \[T_{B_2}':b\mapsto T\left(0\lbrack b^2\rbrack 0\right)0_{2k+2n+2}|_{n+1}^{n+2k+1}.\] If $n$ is odd, we define both $T_{B_1}'$ and $T_{B_2}'$ on $\algA(k+\frac{n+1}{2})$ in analogy with the definition of $T_{B_1}$ and $T_{B_2}$ for even $n$.

If we can show that the maps $T_{B_i}$ are \emph{injective} as well as surjective, by the inclusion-exclusion principle we arrive at the following general recursion:

\begin{align*}a_T(2k) = \hspace{0.1cm}&2a_T(k) + a_T(k+\left\lfloor\tfrac{n}{2}\right\rfloor) + a_T(k+\left\lfloor\tfrac{n+1}{2}\right\rfloor)\\
&-|A_1\cap A_2|-|A_1\cap B_1|-|A_1\cap B_2|-|A_2\cap B_1|-|A_2\cap B_2|-|B_1\cap B_2|\\
&+|A_1\cap B_1\cap B_2|+|A_2\cap B_1\cap B_2|+|A_1\cap A_2\cap B_1|+|A_1\cap A_2\cap B_2|\\
&-|A_1\cap A_2\cap B_1\cap B_2|.
\end{align*}

We will investigate the injectivity of the maps $T_{B_i}$ in the next section, and we will examine the intersections in the subsequent section.

\section{Injectivity}\label{injectivity}

We will now characterize the polynomials for which the maps $T_{B_i}$ are injective on the whole space; for example, if $n$ is even, we will characterize the polynomials for which the maps \[T_{B_i} : (\Z/2)^{k+\frac{n}{2}}\to(\Z/2)^{2k}\] are injective; it follows that the maps \[T_{B_i} : \algA(k+\tfrac{n}{2})\to B_i\] are then bijective.

We will express the maps $T_{B_i}$ in matrix form. The rule $T$ can be written explicitly as \[T(x) = c_0 + c_1x + \ldots + c_nx^n.\] If we wish to multiply this polynomial by another polynomial $S(x) = d_0 + d_1x + \ldots + d_mx^m$, we construct a matrix with $m+1$ columns, of the form

\[
M = \begin{bmatrix}
c_0\\
c_1 & c_0\\
\vdots & c_1\\
c_n & \vdots & & c_0\\
 & c_n & \vdots & c_1\\
 &  &  & \vdots\\
 &  &  & c_n
\end{bmatrix}.
\] The polynomial $(TS)(x)$ is then given by multiplying $M$ by the coefficient vector $(d_0, d_1,\ldots, d_m)$, and expressing the result in the basis $(1, x, x^2, \ldots, x^{n+m})$. Suppose $n$ is even. The action of the map $T_{B_1}$ on a $(k+\frac{n}{2})$-block $b$ can be described by the multiplication of a matrix $\lbrack T_{B_1}\rbrack$ with $b$. We obtain the matrix $\lbrack T_{B_1}\rbrack$ from $M$, first by deleting columns $0, 2, 4,$ etc., in order to express the effect of squaring $b$ and inserting zeros, and then by deleting the first $n$ and last $n+1$ rows of the resulting matrix.

If we define 
\[
C = \begin{bmatrix}
c_{n-1} & c_{n-3} & \cdots & c_1 & 0\\
c_n & c_{n-2} & \cdots & c_2 & c_0
\end{bmatrix},
\] we see that $\lbrack T_{B_1}\rbrack$ is the $2k\times (k+\frac{n}{2})$ matrix in Figure \ref{tb1}.

\begin{figure}[ht]
\centering
\includegraphics[scale = 0.8]{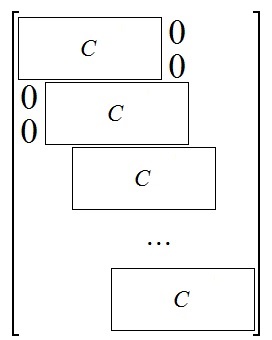}
\caption{The matrix $\lbrack T_{B_1}\rbrack$}
\label{tb1}
\end{figure}

If $n$ is odd, the form of the matrix is also given by Figure \ref{tb1}, but in this case we have 

\[
C = \begin{bmatrix}
c_{n-1} & c_{n-3} & \cdots & c_0\\
c_n & c_{n-2} & \cdots & c_1
\end{bmatrix}.
\] The matrix for $T_{B_1}$ of shape $2k\times (k+\lfloor\frac{n}{2}\rfloor)$. The matrix for $T_{B_2}$ is constructed in an analogous manner, and is of shape $2k\times (k+\lfloor\frac{n+1}{2}\rfloor)$. A chart of the matrices for $T_{B_2}$ is given in Figure \ref{Chart}.

\begin{figure}
\centering
\includegraphics[scale = 0.8]{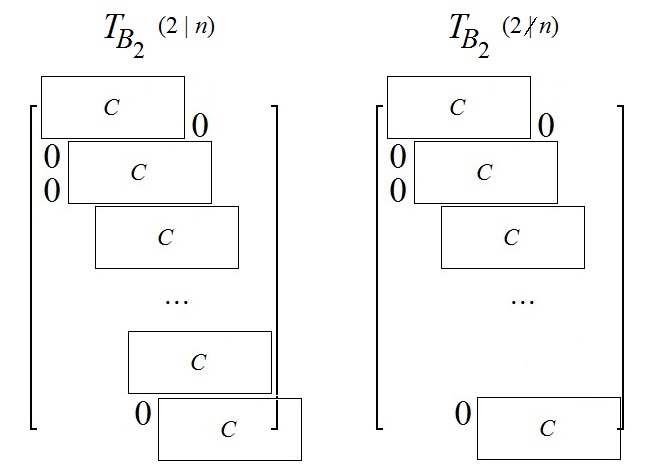}
\caption{Matrices for the transformation $T_{B_2}$.}
\label{Chart}
\end{figure}

In the following, we assume that $n\geq 1$. We say that a polynomial is \emph{suspicious} if there exists $k\geq \lfloor\frac{n}{2}\rfloor$ for which either $T_{B_1}$ or $T_{B_2}$ is not injective (note that here the domains are $(\Z/2)^{k+\frac{n}{2}}$ in the even case, instead of $\algA(k+\frac{n}{2})$).

The reason for the assumption that $k\geq \lfloor\frac{n}{2}\rfloor$ is the following: if $k< \lfloor\frac{n}{2}\rfloor$, we have \[\text{rank} \lbrack T_{B_1}\rbrack\leq 2k< k+\left\lfloor\frac{n}{2}\right\rfloor,\] so that $T_{B_1}$ is definitely not injective.

We now observe that the matrices for the maps $T_{B_i}'$ may be found by precisely the same method as above; in fact, they exhibit the same pattern of translated submatrices, and may be considered enlarged versions of the matrices $\lbrack T_{B_i}\rbrack$. For example, if $n$ is odd, the matrix $\lbrack T_{B_1}'\rbrack$ is obtained from $\lbrack T_{B_1}\rbrack$ by adding the next row and the next column in the pattern.

From this observation it is easy to show that
\begin{prop}
If $T$ is not suspicious, then the maps $T_{B_1}'$ and $T_{B_2}'$ are injective for $k\geq \lfloor\frac{n}{2}\rfloor +1$.
\end{prop}

We will now define \emph{odd} and \emph{even} parts of a polynomial $T$. We write $T(x)=c_0+c_1x+\ldots+c_nx^n \hspace{0.2cm}(c_n\neq 0)$ as before, and define
\begin{align*}o(x) &= c_1x + c_3x^2+\ldots +c_{n-1}x^{\frac{n}{2}}\\
e(x) &= c_0+c_2x+\ldots +c_nx^{\frac{n}{2}}\end{align*} if $n$ is even, and
\begin{align*}o(x) &= c_1x+c_3x^2+\ldots +c_nx^{\frac{n+1}{2}}\\
e(x) &= c_0+c_2x+\ldots +c_{n-1}x^{\frac{n-1}{2}}\end{align*} if $n$ is odd. With these definitions, it follows immediately that $T(x) = \frac{o(x^2)}{x} + e(x^2)$. We are now ready to present a characterization of the nonsuspicious polynomials.

\begin{theorem} Suppose that $T$ is a polynomial in $(\Z/2)[x]$, of degree $n\geq 1$.

\noindent If $n$ is even, then

\emph{i.} $T_{B_1}$ is injective if and only if $c_0\neq 0$ and $(o,e)=1$, and

\emph{ii.} $T_{B_2}$ is injective if and only if $(o,e)=1$.

\noindent If $n$ is odd, then

\emph{iii.} $T_{B_1}$ is injective if and only if $(o,e)=1$, and

\emph{iv.} $T_{B_2}$ is injective if and only if $c_0\neq 0$ and $(o,e)=1$.
\end{theorem}

\begin{cor}
$T$ is nonsuspicious if and only if $c_0\neq 0$ and $(o, e)=1$.
\end{cor}

\begin{proof}
Consider the map $T_{B_1}$, and assume that $n$ is even. We claim that for any $k\geq \frac{n}{2}$, $T_{B_1}$ is injective if and only if $T_{B_1}$ is injective in the special case of $k=\frac{n}{2}$. Sufficiency is clear; necessity follows from the structure of the matrix $\lbrack T_{B_1}\rbrack$, as in the following: we will write $T_{B_1}(k=\frac{n}{2})$ for the operator $T_{B_1}$ in the case where $k=\frac{n}{2}$.

Suppose\[T_{B_1}x = 0,\] where $x=x_0x_1\cdots x_{k+\frac{n}{2} - 1}\in(\Z/2)^{k+\frac{n}{2}}$. Since the submatrix which consists of the first $n$ rows and columns of $T_{B_1}$ is precisely $\lbrack T_{B_1}\rbrack(k=\frac{n}{2})$, we see that \[x_0=x_1=\cdots = x_{n-1} = 0.\] We now observe that the submatrix which consists of the entries in rows $2$ through $n+1$ and columns $1$ through $n$ is \emph{also} precisely $\lbrack T_{B_1}\rbrack(k=\frac{n}{2})$, so that \[x_2 = \cdots = x_n = x_{n+1} = 0.\]

Continuing this process, we conclude that $x=0$, so that $T_{B_1}$ is injective.

We have shown that it is sufficient to analyze the case of $T_{B_1}(k=\frac{n}{2})$. We now take the determinant of $\lbrack T_{B_1}\rbrack(k=\frac{n}{2})$.

We use the fact that the determinant changes only in sign under permutations of the rows and columns. First,

\[\det{\lbrack T_{B_1}\rbrack(k=\tfrac{n}{2})} = \pm \det\begin{bmatrix}
c_{n-1} & c_{n-3} & \cdots & c_1\\
& c_{n-1} & c_{n-3} & \cdots & c_1\\
&&&&\cdots\\
&&&c_{n-1}&c_{n-3}&\cdots &c_1&0\\
c_n & c_{n-2} & \cdots & c_0\\
& c_n & c_{n-2} & \cdots & c_0\\
&&&&\cdots\\
&&&&c_n & c_{n-2} & \cdots & c_0
\end{bmatrix}.\]

By expansion along the last column, we conclude that \[\det{\lbrack T_{B_1}\rbrack(k=\tfrac{n}{2})} = \pm c_0\det\begin{bmatrix}
c_{n-1} & c_{n-3} & \cdots & c_1\\
& c_{n-1} & c_{n-3} & \cdots & c_1\\
&&&&\cdots\\
&&& c_{n-1} & c_{n-3} & \cdots & c_1\\
c_n & c_{n-2} & \cdots & c_2 & c_0\\
& c_n & c_{n-2} & \cdots & c_2 & c_0\\
&&&&\cdots\\
&&&c_n & c_{n-2} & \cdots & c_2 & c_0
\end{bmatrix}.\]

We now reverse the order of the columns, and obtain
\[\det{\lbrack T_{B_1}\rbrack(k=\tfrac{n}{2})} = \pm c_0\det\begin{bmatrix}
&&& 0 & c_1 & \cdots & c_{n-1}\\
&&&&\cdots\\
&0 & c_1 & \cdots & c_{n-1}\\
0 & c_1 & \cdots & c_{n-1}\\
&&& c_0 & c_2 & \cdots & c_n\\
&&&&\cdots\\
& c_0 & c_2 & \cdots & c_n\\
c_0 & c_2 & \cdots & c_n
\end{bmatrix}\]

We now permute the order of the rows, to conclude that
\[\det{\lbrack T_{B_1}\rbrack(k=\tfrac{n}{2})} = \pm c_0\det\begin{bmatrix}
0 & c_1 & \cdots & c_{n-1}\\
& 0 & c_1 & \cdots & c_{n-1}\\
&&&&\cdots\\
&&& 0 & c_1 & \cdots & c_{n-1}\\
c_0 & c_2 & \cdots & c_n\\
& c_0 & c_2 & \cdots & c_n\\
&&&&\cdots\\
&&& c_0 & c_2 & \cdots & c_n
\end{bmatrix}.\]

The latter determinant is precisely the \emph{resultant} of the polynomials $o$ and $e$ defined above. By Corollary 1.8 in \cite{Janson}, the resultant of $o$ and $e$ is nonzero if and only if $(o, e)=1$. We have thus proved part i.

The proofs of the other assertions are similar; to prove part ii we use expansion in the first column; the proof of part iii only requires interchanging rows and columns, and to prove part iv we use expansion in the first and last columns successively.
\end{proof}

We close this section with a corollary which illustrates an interesting connection between the notion of suspiciousness and the algebraic properties of the polynomial in question.

\begin{cor}
Suppose $T$ (as above) is irreducible, with $c_0\neq 0$. Then $T$ is not suspicious.
\end{cor}

\begin{proof}
We will prove the contrapositive. Suppose that $T$ is suspicious, with $c_0\neq 0$, so that $o$ and $e$ have a nontrivial common factor $f$. Then $o(x)=f(x)o_f(x)$ and $e(x) = f(x)o_e(x)$ for some polynomials $o_f, o_e$. It follows that \[T(x)=\frac{o(x^2)}{x}+e(x^2)=f(x^2)\left(\frac{o_f(x^2)}{x}+o_e(x^2)\right).\] Thus $f(x^2)|T(x)$, so that $T$ is reducible.
\end{proof}

\section{Intersections}\label{intersections}

In this section, we investigate the sizes of intersections of the sets $A_1$, $A_2$, $B_1$, and $B_2$. We will consider the case where the elements are of even length, for specificity. We first note that the only element of $A_1\cap A_2$ is the zero string $0^{2k}$. It follows that

\begin{prop}
We have
\[|A_1\cap A_2| = |A_1\cap A_2 \cap B_1| = |A_1\cap A_2\cap B_2| = |A_1\cap A_2\cap B_1\cap B_2| = 1.\]
\end{prop}

In the following, there will always be an implicit dependence on the length $k$; for instance, we may write $|A_1(k)\cap B_1(k)| = |A_1\cap B_1|$ unless the length is explicitly required.

We will now introduce several transformations that will be used in the proofs of the next few theorems. These transformations are all defined on $(\Z/2)^{2k}$. We first define \[S_0 = \begin{bmatrix}c_n & c_{n-1} & \cdots & c_0\\
& c_n & c_{n-1} & \cdots & c_0\\
&& c_n & c_{n-1} & \cdots & c_0\\
&&&&&\cdots\\
&&&& c_n & c_{n-1} & \cdots & c_0\end{bmatrix}.\] We now suppose that $n$ is even, and define
\[S_1 = \begin{bmatrix}c_n & c_{n-1} & \cdots & c_0\\
0 & 0 & 0 & \cdots\\
0 & 0 & c_n & c_{n-1} & \cdots & c_0\\
0 & 0 & 0 & 0 & 0 & \cdots\\
&&&&&&\cdots\\
&&&&& c_n & c_{n-1} & \cdots & c_0 & 0\end{bmatrix}\] and
\[S_2 = \begin{bmatrix}0 & 0 & \cdots\\
0 & c_n & c_{n-1} & \cdots & c_0\\
0 & 0 & 0 & 0 & \cdots\\
0 & 0 & 0 & c_n & c_{n-1} & \cdots & c_0\\
&&&&&&\cdots\\
&&&&& c_n & c_{n-1} & \cdots & c_0\end{bmatrix}.\] If $n$ is odd, the transformations $S_1$ and $S_2$ are defined in the same way, except that the zero entry in the lower right-hand corner appears in $S_2$, but not $S_1$: the reason for this zero entry is that the transformations are defined on $(\Z/2)^{2k}$, and the last row must therefore contain an even number of entries. We also require that $k\geq\left\lfloor\frac{n}{2}\right\rfloor+1$ when using these transformations; otherwise, these matrices would not contain all the coefficients $c_0,\ldots,c_n$. In order to avoid trivialities, we will also assume that $T$ is of positive degree.

We now turn to the intersection $B_1\cap B_2$.

\begin{theorem}\label{intb1b2}
If $T$ is a polynomial in $(\Z / 2)[x]$, of degree $n\geq 1$, and with $c_0\neq 0$, then the size of $B_1(k)\cap B_2(k)$ decreases monotonically for $k\geq n$, and, hence, is independent of $k$ for sufficiently large $k$.

\end{theorem}

\begin{proof}
We will suppose that $n$ is even. Precisely the same recursions described below hold if $n$ is odd, so the argument is identical. The proof will be based upon the following observation: \emph{All blocks in $B_1\cap B_2$ are in the kernel of the transformation $S_0$}.

To see this, suppose $b = b_0b_1\cdots b_{2k-1}\in B_1\cap B_2$. Since $b\in B_1$, we have $b = T_{B_1}x$ for some $x\in \algA(k+\frac{n}{2})$. Consider the product $S_0T_{B_1}$. A direct computation shows that the matrix $\lbrack S_0T_{B_1}\rbrack$ is given by

\[\begin{bmatrix}c_n & c_{n-1} & \cdots & c_0\\
& c_n & c_{n-1} & \cdots & c_0\\
&& c_n & c_{n-1} & \cdots & c_0\\
&&&&&\cdots\\
&&&& c_n & c_{n-1} & \cdots & c_0\end{bmatrix}\begin{bmatrix}c_{n-1} & c_{n-3} & \cdots & c_1 & 0\\
c_n & c_{n-2} & \cdots & c_2 & c_0\\
& c_{n-1} & c_{n-3} & \cdots & c_1 & 0\\
& c_n & c_{n-2} & \cdots & c_2 & c_0\\
&&&&\cdots\\
&&&&\cdots\\
&&& c_{n-1} & c_{n-3} & \cdots & c_1 & 0\\
&&& c_n & c_{n-2} & \cdots & c_2 & c_0\end{bmatrix}\]

\[= \begin{bmatrix}0 & \cdots\\
c_n^2 & c_{n-1}^2 & \cdots & c_0^2 & 0 & \cdots && 0\\
0 & \cdots\\
0 & c_n^2 & c_{n-1}^2 & \cdots & c_0^2 & 0 & \cdots & 0\\
&&&&&\cdots\\
0 & \cdots && 0 & c_n^2 & c_{n-1}^2 & \cdots & c_0^2
\end{bmatrix} = \begin{bmatrix}0 & \cdots\\
c_n & c_{n-1} & \cdots & c_0 & 0 & \cdots && 0\\
0 & \cdots\\
0 & c_n & c_{n-1} & \cdots & c_0 & 0 & \cdots & 0\\
&&&&&\cdots\\
0 & \cdots && 0 & c_n & c_{n-1} & \cdots & c_0\end{bmatrix}.\] It follows that the $i$th coordinate of $S_0T_{B_1}x$ is $0$ if $i = 0, 2, \ldots$ is even.

The situation is precisely analogous for $T_{B_2}$; in this case we conclude that odd coordinates of $S_0T_{B_2}y$ are zero for $y \in \algA(k+\frac{n}{2})$. It follows that $S_0b=0$, so that $b\in \ker S_0$.

From this we see that $c_nb_0 + c_{n-1}b_1 + \ldots + c_0b_n = 0$, $c_nb_1 + c_{n-1}b_2 + \ldots + c_0b_{n+1} = 0$, etc. Thus, $b_n$ is determined by $b_0$, $b_1$, $\ldots$, $b_{n-1}$. Moreover, $b_{n+1}$ is determined by $b_1$, $b_2$, $\ldots$, $b_n$, and so on. It follows that $b$ is determined entirely by the accessible block $b_0b_1\cdots b_{n-1}$. We employ similar methods to conclude that this also holds for blocks of odd length.

Suppose $n \leq j < k$. From the above we see that every block in $B_1 \cap B_2$ of length $r\geq n$ is generated by a block of length $n$. The map that assigns to a block of length $k$ the block of length $j$ with the same generator is injective, so that \[|B_1(j)\cap B_2(j)| \geq |B_1(k)\cap B_2(k)|.\] In particular, the sequence $|B_1(k)\cap B_2(k)|$ is nonincreasing for $k\geq n$; since $|B_1(k)\cap B_2(k)|\geq 1$, it follows that the sequence is eventually constant.

\end{proof}

\begin{cor}
The conclusions of Theorem \ref{intb1b2} also hold for $A_1\cap B_1\cap B_2$ and $A_2\cap B_1\cap B_2$.
\end{cor}

\begin{proof}
Since $A_1\cap B_1\cap B_2,~A_2\cap B_1\cap B_2\subseteq B_1\cap B_2$, the recursions described above also apply to these sets. The result follows in the same manner.
\end{proof}

We now consider the intersections $A_1\cap B_1$ and $A_2\cap B_2$.

\begin{theorem}\label{inta1b1}
Suppose $T$ is a polynomial of degree $n\geq 1$ in $(\Z/2)[x]$, with $c_0\neq 0$. Then $|A_1(k)\cap B_1(k)|$ decreases monotonically for $k\geq n$. If $n$ is odd, and $T$ is as above, or if $n$ is even, and we assume additionally that $T$ has at least one nonzero coefficient $c_i$ with $0<i<n$, then $|A_2(k)\cap B_2(k)|$ decreases monotonically for $k\geq n+1$. In particular, the sizes of the intersections $A_1\cap B_1$ and $A_2\cap B_2$ are independent of $k$ for sufficiently large $k$.
\end{theorem}

We observe that all the conditions imposed on $T$ are certainly satisfied if $T$ is not suspicious.

\begin{proof}
We first consider the case of $A_1\cap B_1$. Suppose that $n$ is even. We first note that all blocks in $A_1\cap B_1$ are in the kernel of $S_1$: indeed, we observe that the entries in odd rows of the matrix $\lbrack ST_{A_1}\rbrack$ consist entirely of zeros, and, as in the proof of Theorem \ref{intb1b2}, the even rows (starting with row $0$) consist entirely of zeros.

Suppose $b$ is an arbitrary block in $A_1\cap B_1$. Then we have $b=T_{A_1}x = T_{B_1}y$ for some $x\in\algA(k)$, $y\in \algA(k+\frac{n}{2})$. By the above reasoning, we see that $Sb = S(T_{A_1}x) = S(T_{B_1}y) = 0$. Write $b=b_0\cdots b_{2k-1}$. Since $b\in \ker S_1$ and $b\in A_1$, we have 
\begin{align*}&b_0 + c_{n-1}b_1 + \ldots + c_1b_{n-1} + b_n = 0,\\
&b_{n+1}=0,\\
&b_2 + c_{n-1}b_3 + \ldots + c_1b_{n+1} + b_{n+2} = 0,\\
&b_{n+3}=0,\end{align*} etc., so that $b_0,\ldots,b_{n-1}$ determine $b$. But in fact, since $b_1 = b_3 = \cdots = b_{n-1} = 0$, it follows that $b$ is determined entirely by $b_0,b_2,\ldots,b_{n-2}$. Moreover, since $b = b_00b_20\cdots b_{2k-2}0 = T_{A_1}x$, we must have $x = b_0b_2\cdots b_{2k-2}$, so that $b_0b_2\cdots b_{n-2}$ is accessible. The rest of the proof of this case is identical to the argument in Theorem \ref{intb1b2}.

Now suppose that $n$ is odd. In this case, we see that a given block $b=b_0\cdots b_{2k-1}$ is in $\ker S_1$ exactly as in the previous case; hence the associated recursions hold here as well. We now consider three cases.

First, suppose that $c_2 = \cdots = c_{n-1} = 0$. In this case, the recursions become \begin{align*}
&b_0+b_n=0,\\
&b_2+b_{n+2}=0,\\
&\ldots\\
&b_{2k-1-n}+b_{2k-1}=0.\end{align*} Since $n$ is odd and $b\in A_1$, we see that $b_1=b_3=\cdots=b_n=\cdots=b_{2k-1}=0$; by the recursions above, we see that \[b = 0\cdots 0b_{2k-n+1}0\cdots 0b_{2k-2}0.\] By arguing in analogy with the even case, we see that $b_{2k-n+1}\cdots b_{2k-2}$ is accessible; since we know that $b$ is determined entirely by $b_{2k-n+1}\cdots b_{2k-2}$, the result follows in this case by reasoning similar to that in Theorem \ref{intb1b2}.

Second, suppose that there exists an odd $i$ with $0<i<n$, and for which $c_i\neq 0$. Let $j$ be the smallest such $i$. The recursions now take the form
\begin{align*}
&b_0 + c_{n-2}b_2 + \ldots + c_jb_{n-j} = 0,\\
&b_2 + c_{n-2}b_4 + \ldots + c_jb_{n-j+2} = 0,\\
&...\\
&b_{2k-n-1}+c_{n-2}b_{2k-n+1} + \ldots + c_jb_{2k-j-1} = 0.
\end{align*} The remaining entries $b_{2k-j+1}$, $\ldots$, $b_{2k-2}$ are not determined by the recursions. We conclude that $b$ is determined by $b_0,b_2,\ldots,b_{n-j-2},b_{2k-j+1},b_{2k-j+3},\ldots,b_{2k-2}$; the result follows as above.

Finally, suppose that $c_i=0$ for all odd $i$ with $0<i<n$, but there is some even $i$ with $c_i\neq 0$, $0<i<n$. We now consider the automaton that results from reversing the order of the coefficients in the transition rule; for example, if the rule (expressed as a block) is $1011$, we consider the rule $1101$.

The automaton that results from this operation is a mirror-image of the original; its accessible blocks are reversals of the accessible blocks in the original. Moreover, since $n$ is odd, our hypothesis guarantees that the reversed transition rule has a nonzero odd coefficient $i$ with $0<i<n$. We have thus reduced this case to the previous one.

The case of $A_2\cap B_2$ is analogous to that of $A_1\cap B_1$; the reasoning is more or less the same, with the transformation $S_2$ in place of $S_1$, and with the reasoning for the odd and even cases reversed from what it was above. We also note that the even case for $A_2\cap B_2$, which is analogous to the odd case for $A_1\cap B_1$, only contains two cases, in view of the hypothesis that the transition rule have a nonzero coefficient in the interior when $n$ is even.
\end{proof}

\begin{theorem}
Let $T$ be a polynomial of degree $n\geq 1$ in $(\Z/2)[x]$, with $c_0\neq 0$. If $n$ is even, suppose that there exists an odd $j$ with $0<j<n$, $c_j\neq 0$. Then the sequence $|A_2(k)\cap B_1(k)|$ decreases monotonically for $k\geq n$, and the sequence $|A_1(k)\cap B_2(k)|$ decreases monotonically for $k\geq n+1$. In particular, these sequences are independent of $k$ for sufficiently large $k$.
\end{theorem}

\begin{proof}
We first consider $A_2\cap B_1$. We note that if $b=b_0\cdots b_{2k-1}\in A_2\cap B_1$, then $b\in\ker S_1$. Suppose that $n$ is even. Then \[b_0=b_2=\cdots =b_n=\cdots =b_{2k-2}=0.\] Suppose that $j$ satisfies $0<j<n$, $2\!\!\!\not| j$, $c_j\neq 0$, and is minimal. In analogy with the proof of Theorem \ref{inta1b1}, the recursions become
\begin{align*}
&c_{n-1}b_1 + c_{n-3}b_3 + \ldots + c_jb_{n-j} = 0,\\
&c_{n-1}b_3 + c_{n-3}b_5 + \ldots + c_jb_{n-j+2} = 0,\\
&...\\
&c_{n-1}b_{2k-n-1}+c_{n-3}b_{2k-n+1} + \ldots + c_jb_{2k-j-2}=0.
\end{align*} It follows that $b$ is determined by $b_1b_3\cdots b_{n-j-2}$ and $b_{2k-j+2}\cdots b_{2k-1}$; the result follows as above.

Now suppose that $n$ is odd. Then \[b_0 = b_2 = \cdots = b_{n+1} = \cdots = b_{2k-2} = 0,\] and $b$ is determined by $b_1b_3\cdots b_{n-2}$ as in Theorem \ref{intb1b2}.

The argument for $A_1\cap B_2$ is similar; in this case we use the map $S_2$ instead of $S_1$.
\end{proof}

We combine the results of sections \ref{injectivity} and \ref{intersections} with the general recursion presented at the end of section \ref{recursion_formulas} to obtain a recursion for even-length blocks; a recursion for odd-length blocks follows similarly. We summarize these results in the following theorem:

\begin{theorem}\label{main}
Suppose T is a polynomial of degree $n\geq 1$ that is not suspicious. Then for sufficiently large $k$, we have
\[a_T(2k) = 2a_T(k) + a_T(k+\left\lfloor\tfrac{n}{2}\right\rfloor) + a_T(k+\left\lfloor\tfrac{n+1}{2}\right\rfloor) + C_\cap\] and
\[a_T(2k+1) = a_T(k) + a_T(k+1) + a_T(k+\left\lfloor\tfrac{n+1}{2}\right\rfloor) + a_T(k + \left\lfloor\tfrac{n}{2}\right\rfloor + 1) + C_\cap,\] where $C_\cap$ is a constant dependent only on $T$.
\end{theorem}

\section{Powers of the Transition Rule}

So far, our analysis has been based heavily upon the injectivity of the maps $T_{B_i}$. It turns out, however, that some suspicious polynomials obey recursions similar to those described above. To examine this phenomenon more closely, we introduce the \emph{order} of a recursion.

Suppose $p$ is prime. We say that the automaton $A_p(I;T)$ \emph{satisfies a recursion of order n} if there exist a constant $C$ and integers $n, K$ such that
\[a_T(k) = \Sum_{j=0}^{p-1}\Sum_{r=0}^{p-1}a_T\left(\left\lfloor\frac{k+jn+r}{p}\right\rfloor\right)+C\] for all $k\geq K$.

In the modulo 2 case we considered above, the degree of the polynomial and the order of the recursion were the same; in general this need not be the case.

In this section we will show that if the automaton $A_p(c;T)$ (with a constant initial state) satisfies a recursion of order $r$, then $A_p(c; T^n)$ satisfies a recursion of order $rp^s$, where $s$ is the largest integer such that $p^s \mid n$.

We will say that an automaton is \emph{trivial} if its transition rule $T$ has at most one nonzero coefficient. Such rules can only translate the initial state or multiply it by a constant. The following proposition shows that non-trivial automata use the entire alphabet available to them.

\begin{prop}
Suppose $p$ is prime. If the automaton $A_p(I; T)$ is non-trivial, then all the symbols $0, 1, \ldots, p - 1$ are accessible; that is, $a_T(1) = p$.
\label{nontriv}
\end{prop}
\begin{proof}
Write $T(x) = a_0 + a_1x + \ldots + a_nx^n$. Since $a_{xT(x)}(k) = a_{T(x)}(k)$ for all $k$, we may assume that $a_0 \neq 0$. We also suppose (since the automaton is nontrivial) that $a_d \neq 0$ for some $d > 0$; we choose $d$ so as to be minimal. Then the coefficient of $x^d$ in the expansion of $T^r$ is $ra_0^{r-1}a_d$. Since $p$ is prime and $a_0 \neq 0$, we have $a_0^{k(p-1)}\equiv 1\pmod p$ (note that the multiplicative group $(\Z/p)^\times$ of nonzero integers modulo $p$ is cyclic, so $a_0$ has a finite order which divides $p-1$). Thus, taking $r = k(p-1) + 1$, we see that $ra_0^{r-1}a_d = (1+k(p-1))a_0^{k(p-1)}a_d = (1+k(p-1))a_d$. Since $(\Z/p)^+$ (the \emph{additive} group of integers modulo $p$) is of prime order, it is cyclic and is generated by every nonzero element. Since $a_d, p - 1 \neq 0$, we see that $\{(1+k(p-1))a_d:k\geq 0\} = (\Z/p)^+$, so that the coefficient of $x_d$ assumes all values in $\Z/p$. This completes the proof.
\end{proof}

\begin{prop}
Suppose $p$ is prime, $c,d\in\Z/p$, and the automaton $A_p(d, T)$ is non-trivial. Then if $b\in\algA(k)$, we have $c\cdot b\in\algA(k)$.
\label{orbit}
\end{prop}

\emph{Note:} in particular, if we consider the operation of multiplication by a constant as an action of $\Z/p$ on $\algA(k)$, then this proposition implies that all orbits of blocks in $\algA(k)$ are contained in $\algA(k)$: we have \[(\Z/p)(\algA(k)) = \algA(k)\hspace{1cm}(k\geq 1).\]

\begin{proof}
Since $A_p(d; T)$ is non-trivial, Proposition \ref{nontriv} shows that $c\cdot d\in\algA(1)$. Suppose that $c\cdot d$ appears in line $r$. We first note that $T^p(s) \equiv T(s^p) \pmod p$ for any polynomial $s$, since $p$ is prime. Thus if $j\geq 1$, line $p^j r$ includes the block $0^j(c\cdot d)0^j$. The block $b$ appears in some line, say $r'$. If $j > (\deg T)r'$, then lines $0, 1, \ldots, r'$ of the automaton, multiplied by $c$, appear in the lines $p^j r, \ldots, p^j r + r'$. In particular, we can take $j = (\deg T)r' + 1$; then $c\cdot b$ appears in row $p^{(\deg T)r'+1}r+r'$. This completes the proof.
\end{proof}

In the following theorem, we show that taking the transition rule to powers relatively prime to the modulo does not change the collection of accessible blocks.

\begin{theorem}
Suppose that $p$ is prime, $(p, n) = 1$, $c \in \Z/p$, and suppose that $A_p(c;T)$ is non-trivial. Let $\algA_n(k)$ denote the set of accessible blocks of length $k$ associated to $A_p(c; T^n)$. Then for all $k$, we have \[\algA_1(k) = \algA_n(k),\] and in particular, \[a_T(k) = a_{T^n}(k).\]\label{relprime}\end{theorem}
\begin{proof}
Since $a_{xT(x)}(k) = a_{T(x)}(k)$, we will assume that $T$ has a nonzero constant coefficient. We first note that line $k$ of $A_p(c; T^n)$ is the same as line $kn$ of $A_p(c; T)$, since at each stage $A_p(c; T^n)$ applies the transition rule $n$ times. From this it clearly follows that $\algA_n(k) \subseteq \algA_1(k)$ for all $k$. To show that $\algA_n(k) \supseteq \algA_1(k)$, suppose that $b$ is a block of length $k$ in $A_p(c; T)$, appearing on some line $r$. We will show that $b$ appears on a line $L\equiv 0\pmod n$. This is clear if $r\equiv 0\pmod n$; we will thus assume that $r\not\equiv 0 \pmod n$.

As in the proof of Proposition \ref{orbit}, we have $T(s)^p \equiv T(s^p) \pmod p$ for any polynomial $s$. Thus if $j\geq 1$ and $T(x) = t_0 + t_1x + \ldots + t_nx^n$, line $1$ of $A_p(c; T)$ is given by $(c\cdot t_0)\cdots(c\cdot t_n)$, so that line $p^j$ is given by $(c\cdot t_0)0^j(c\cdot t_1)0^j\cdots 0^j(c\cdot t_n)$. Thus, as in the proof of Proposition \ref{orbit}, if $j>nr$, lines $0, \ldots, r$ of the automaton $A_p((c\cdot t_0); T)$ appear in lines $p^j, \ldots, p^j+r$ of $A_p(c; T)$.

Since $(p,n)=1$, we have $p^{\phi(n)}\equiv 1\pmod n$ by the Euler-Fermat theorem. In particular, $p$ is of finite order $m$. Thus, there exists $s_1$ such that $p^{s_1m}\equiv 1\pmod n$ and $s_1m>rn$, so that $t_0\cdot b$ appears in line $r + p^{s_1m}$ and $r+p^{s_1m} \equiv r+1\pmod n$. If $r+1\equiv 0\pmod n$, then we have $t_0\cdot b\in\algA_n(k)$.

Otherwise, we repeat the above reasoning with $r$ replaced by $r+p^{s_1m}$: we know that $t_0\cdot b$ appears in row $r+p^{s_1m}+p^j$ if $j>n(r+p^{s_1m})$. Thus there exists $s_2$ such that $p^{s_2m}\equiv 1\pmod n$ and $s_2m > n(r+p^{s_1m})$. It follows that $t_0\cdot b$ appears in row $r+p^{s_1m}+p^{s_2m}$, and $r+p^{s_1m}+p^{s_2m}\equiv r+2\pmod n$. If $r+2\equiv 0\pmod n$, then $t_0\cdot b\in\algA_n(k)$; otherwise, we proceed in this manner until $r+i\equiv 0\pmod n$ for some $i$ (Note that at most finitely many steps are necessary).

We have shown that $t_0\cdot b\in\algA_n(k)$. Since $\Z/p$ is a field, $t_0$ has an inverse. By applying Proposition \ref{orbit} to the automaton $A_p(c;T^n)$, we see that $t_0^{-1}\cdot(t_0\cdot b)\in\algA_n(k)$. 
We have shown that $\algA_1(k) = \algA_n(k)$. The conclusion follows.
\end{proof}

\begin{theorem}
Suppose $p$ is prime and $0\leq r<p$. Then we have \[a_{T^p}(pk+r) = (p-r)a_T(k)+ra_T(k+1)+1-p\] for all $k\geq 1$.
\end{theorem}

\begin{proof}
First, suppose that $r\geq 1$. We will use the notation $\algA_n$ of Theorem \ref{relprime}. In view of the identity $T(s^p)\equiv T(s)^p\pmod p$, we see that the accessible coefficient blocks of length $pk+r$ must belong to one of the following sets:
\begin{alignat*}{2}
A_1 &= \{x_00^{p-1}x_10^{p-1}\cdots x_{k-1}0^{p-1}x_k0^{r-1}&&:x_0\cdots x_{k}\in\algA_1(k+1)\}\\
A_2 &= \{0x_00^{p-1}x_10^{p-1}\cdots x_{k-1}0^{p-1}x_k0^{r-2}&&:x_0\cdots x_{k}\in\algA_1(k+1)\}\\
\cdots&\\
A_r &= \{0^{r-1}x_00^{p-1}x_10^{p-1}\cdots x_{k-1}0^{p-1}x_k&&:x_0\cdots x_{k}\in\algA_1(k+1)\}\\
A_{r+1} &= \{0^rx_00^{p-1}x_10^{p-1}x_20^{p-1}\cdots x_{k-1}0^{p-1}&&:x_0\cdots x_{k-1}\in\algA_1(k)\}\\
\cdots&\\
A_p &= \{0^{p-1}x_00^{p-1}x_10^{p-1}x_20^{p-1}\cdots x_{k-1}0^r&&:x_0\cdots x_{k-1}\in\algA_1(k)\}
\end{alignat*}
Thus $\algA_p(pk+r) = A_1\cup\cdots\cup A_p$. Note that for $i \leq r$ the mappings $m_i: \algA_1(k+1)\to A_i$ defined by $m_i:x_0\cdots x_k\mapsto 0^{i-1}x_00^{p-1}x_10^{p-1}\cdots x_{k-1}0^{p-1}x_k0^{r-i}$ are bijective, and the same is true of the analogous mappings $m_i:\algA_1(k)\to A_i\hspace{0.1cm}(i>r)$. It follows that
\[|A_i| = \begin{cases}a_T(k+1) & i\leq r\\a_T(k) & i>r.\end{cases}\] Moreover, it is clear that all pairwise intersections of the sets $A_i$ contain only the string $0^{pk+r}$. The inclusion-exclusion principle thus gives
\begin{align*}
a_{T^p}(pk+r) = |\algA_p(pk+r)| &= |A_1| + \ldots + |A_p| + \Sum_{2\leq |J|\leq p}(-1)^{|J|-1}\binom{p}{|J|}\\
&= (p-r)a_T(k)+ra_T(k+1) + \Sum_{2\leq |J|\leq p}(-1)^{|J|-1}\binom{p}{|J|}\\
&= (p-r)a_T(k)+ra_T(k+1) - \Sum_{i=2}^p(-1)^i\binom{p}{i}\\
&= (p-r)a_T(k)+ra_T(k+1) + (1-p) - \Sum_{i=0}^p(-1)^i\binom{p}{i}\\
&= (p-r)a_T(k)+ra_T(k+1) + (1-p) - (1+(-1))^p\\
&= (p-r)a_T(k)+ra_T(k+1) + (1-p).
\end{align*}
The case $r=0$ follows by precisely analogous reasoning -- in particular, we can use the same sets $A_i$ as above, if we consider the symbol $0^{-1}$ as ``backspace;'' these sets then all have cardinality $a_T(k)$. This completes the proof.
\end{proof}

\begin{cor}
If $A_p(I;T)$ satisfies a recursion of order $n$, then $A_p(I; T^p)$ satisfies a recursion of order $pn$.
\label{recur}
\end{cor}

\begin{proof}
From the last theorem, we have \[a_{T^p}(k) = \Sum_{i=0}^{p-1}a_T\left(\left\lfloor\frac{k+i}{p}\right\rfloor\right)+1-p\hspace{0.4cm}\text{ for }k\geq p.\]

If $C$ and $K$ are as in the definition above, then for $k\geq p(K+1)$ we have $k\geq p$, $\lfloor\frac{k+i}{p}\rfloor\geq K$, so that
\begin{align*}
a_{T^p}(k) &= \Sum_{i=0}^{p-1}\Sum_{j=0}^{p-1}\Sum_{r=0}^{p-1}a_T\left(\left\lfloor\frac{\left\lfloor\frac{k+i}{p}\right\rfloor+jn+r}{p}\right\rfloor\right)+Cp+1-p\\
&= \Sum_{i=0}^{p-1}\Sum_{j=0}^{p-1}\Sum_{r=0}^{p-1}a_T\left(\left\lfloor\frac{\left\lfloor\frac{k+jpn+i}{p}\right\rfloor+r}{p}\right\rfloor\right)+Cp+1-p\\
&= \Sum_{i=0}^{p-1}\Sum_{j=0}^{p-1}\left\lbrack a_{T^p}\left(\left\lfloor\frac{k+jpn+i}{p}\right\rfloor\right)-(1-p)\right\rbrack+Cp+1-p\\
&= \Sum_{i=0}^{p-1}\Sum_{j=0}^{p-1} a_{T^p}\left(\left\lfloor\frac{k+jpn+i}{p}\right\rfloor\right)+Cp+(1-p)(1-p^2)
\end{align*}
so that $a_{T^p}$ satisfies a recursion of order $pn$. The conclusion follows.
\end{proof}

We can now combine the above results to give the following:

\begin{theorem}
Suppose $p$ is prime, $c\in\Z/p$, and $n\in\Z^+$. Let $s$ be the largest integer such that $p^s \mid n$. Then if $A_p(c;T)$ satisfies a recursion of order $r$, $A_p(c; T^n)$ satisfies a recursion of order $rp^s$.
\end{theorem}

\begin{proof}
Since $s$ is maximal, we can write $n=p^sm$, where $(p,m)=1$. It follows that $a_{T^n}(k) = a_{T^{p^s}}(k)$ for all $k$, by Theorem \ref{relprime}, and repeated application of Corollary \ref{recur} shows that $A_p\left(c; T^{p^s}\right)$ satisfies a recursion of order $rp^s$. The conclusion follows.
\end{proof}

\section{Generating Functions}

In this section we will investigate generating functions for the sequences $a_T(k)$ in the modulo $2$ case. We will derive functional equations which will be useful in the next section.

Suppose $T\in (\Z/2)[x]$ is of degree $n$, and $a_T(k)$ satisfies the recursions in Theorem \ref{main}, for all $k\geq N$. Then for complex $z$ with $|z|<\frac{1}{2}$, we define $$f_T(z)=\Sum_{k=2N}^\infty{a_T(k)z^k}.$$ Note that $a_T(k)\leq 2^k$, so that the right-hand expression is defined.

\begin{theorem}
Let $0<|z|<\frac{1}{2}$. Then $$f_T(z)=P_T(z)+\frac{C_\cap z^{2N}}{1-z}+\frac{1}{z^{n+1}}(1+z^n)(1+z)^2f_T(z^2),$$ where $P_T(z)$ is a polynomial.

\end{theorem}

\begin{proof}
We will again write $a(k)$ for $a_T(k)$ in the proof, and we will assume that $n$ is even, for specificity (the case of odd $n$ is analogous). We have

$$f_T(z)=\Sum_{k=2N}^\infty{a(k)z^k}=\Sum_{k=N}^\infty{a(2k)z^{2k}}+z\Sum_{k=N}^\infty{a(2k+1)z^{2k}}.$$ Using the recursions for $a_T(k)$, we obtain \begin{align*}
f_T(z)=&\Sum_{k=N}^\infty{(2a(k)+2a(n/2+k))z^{2k}}\\
&+z\Sum_{k=N}^\infty{(a(k+1)+a(k)+a(n/2+k)+a(n/2+k+1))z^{2k}}\\
&+C_\cap(1+z)\Sum_{k=N}^\infty{z^{2k}}.
\end{align*}
Therefore, \begin{align*}
f_T(z)=&2\Sum_{k=N}^\infty{a(k)z^{2k}}+\frac{2}{z^n}\Sum_{k=N+n/2}^\infty{a(k)z^{2k}}+\frac{1}{z}\Sum_{k=N+1}^\infty{a(k)z^{2k}}\\
&+z\Sum_{k=N}^\infty{a(k)z^{2k}}+\frac{1}{z^{n-1}}\Sum_{k=N+n/2}^\infty{a(k)z^{2k}}+\frac{1}{z^{n+1}}\Sum_{k=N+n/2+1}^\infty{a(k)z^{2k}}\\
&+C_\cap(1+z)\left(\frac{1}{1-z^2}-\Sum_{k=0}^{N-1}{z^{2k}}\right).
\end{align*}
Collecting terms, we have \begin{align*}
f_T(z)=&2\Sum_{k=N}^{2N-1}{a(k)z^{2k}}+\frac{2}{z^n}\Sum_{k=N+n/2}^{2N-1}{a(k)z^{2k}}+\frac{1}{z}\Sum_{k=N+1}^{2N-1}{a(k)z^{2k}}\\
&+z\Sum_{k=N}^{2N-1}{a(k)z^{2k}}+\frac{1}{z^{n-1}}\Sum_{k=N+n/2}^{2N-1}{a(k)z^{2k}}+\frac{1}{z^{n+1}}\Sum_{k=N+n/2+1}^{2N-1}{a(k)z^{2k}}\\
&+\left(2+\frac{2}{z^n}+\frac{1}{z}+z+\frac{1}{z^{n-1}}+\frac{1}{z^{n+1}}\right)f_T(z^2)+\frac{C_\cap(1+z)}{1-z^2}-C_\cap\Sum_{k=0}^{2N-1}z^{k}\\
=&~P_T(z)+\frac{z^{2N}C_\cap}{1-z}+\frac{1}{z^{n+1}}(1+z^n)(1+z)^2f_T(z^2),
\end{align*}
where \begin{align*}
P_T(z)=&~2\Sum_{k=N}^{N+n/2-1}a_T(k)z^{2k}+\left(2+\frac{2}{z^n}\right)\Sum_{k=N+n/2}^{2N-1}a_T(k)z^{2k}\\
&+\left(z+\frac{1}{z}\right)\Sum_{k=N+1}^{2N-1}a_T(k)z^{2k}+\left(\frac{1}{z^{n-1}}+\frac{1}{z^{n+1}}\right)\Sum_{k=N+n/2+1}^{2N-1}a_T(k)z^{2k}\\
&+(a_T(N)+a_T(N+n/2))z^{2N+1}.
\end{align*}
\end{proof}

\section{Asymptotic Behavior of $a_T(k)/k^2$}
For the discussion of asymptotic behavior, we will work in a much more general framework. We will consider a generating function $\phi$ which is assumed to satisfy a general functional equation. The main result we derive in this section will include the case of $a_T(k)$.

We shall require the following facts about power series.

(a) If $R$ is the radius of convergence of the power series, then in $|z|<R$ the sum of the series is analytic and its derivative has the same radius of convergence (see \cite{Ahlfors}, Ch.~2, \S 2.4, Theorem 2(iii)).

(b) If $f$ has a power series development in a disk, then the coefficients are uniquely determined (see \cite{Ahlfors}, p. 40).

Let $p$ be prime and let $D=\{z\in\mathbb C:|z|<1/p\}$. For $|z|<1$, let $\frac{1}{\lambda(z)}=\sum_{k=0}^\infty\gamma(k)z^k$, where $\gamma(0)=1$ and $\gamma(k)=Ck^2+f(k)$, where $C>0$ is constant and $\lim_{k\to\infty}\frac{f(k)\log_p k}{k^2}=0$. Let $\phi:D\to\mathbb C$ be a function given by the power series expression $$\phi(z)=\Sum_{k=1}^\infty{\alpha_kz^k},$$ where $\alpha_k\leq p^k$, and assume that $\phi$ satisfies $$\lambda(z)\phi(z)=R(z)+\lambda(z^p)\phi(z^p),$$ where $R:\mathbb C\to\mathbb C$ is a polynomial with $R(1)=0$. Note that $R(0)=0$.
\begin{prop}
We have $$\phi(z)=\frac{1}{\lambda(z)}\Sum_{k=0}^\infty{R\left(z^{p^k}\right)}.$$
\end{prop}
\begin{proof}
We have that $\lambda(z)\phi(z) = R(z)+\lambda(z^p)\phi(z^p)$. Iterating this equation gives $\lambda(z)\phi(z) = R(z)+R(z^p)+\cdots+R\left(z^{p^k}\right)+\lambda\left(z^{p^{k+1}}\right)\phi\left(z^{p^{k+1}}\right)$. We now note that since $\lambda$ and $\phi$ are analytic and hence continuous in $D$, we have $$\lim_{k\to\infty}\lambda\left(z^{p^{k+1}}\right)\phi\left(z^{p^{k+1}}\right)=r(0)\phi(0)=0.$$ Thus, $\lambda(z)\phi(z)=\sum_{k=0}^\infty{R\left(z^{p^k}\right)}.$ The result follows.
\end{proof}

We now use the above proposition to develop an explicit formula for the coefficients $\alpha_k$ for sufficiently large $k$.
\begin{theorem}
\label{coeffrep}
Set $m=\deg R$, and write $R(z)=\sum_{j=1}^m{c_jz^j}$. 
Then for $k\geq m$,
$$\alpha_k=\Sum_{j=1}^m\Sum_{t=0}^{\left\lfloor\log_p \frac{k}{j}\right\rfloor}c_j\gamma({k-jp^t}).$$
\end{theorem}

\begin{proof}
We have $$\phi(z)=\Sum_{q=0}^\infty\gamma(q)z^q\Sum_{k=0}^\infty{R\left(z^{p^k}\right)}=\Sum_{j=1}^m\Sum_{q=0}^\infty{\gamma(q)z^q}\Sum_{k=0}^\infty{z^{j\cdot {p^k}}}.$$ We now note that $|f(q)|<q^2$, for otherwise we would not have $f(q)(\log_p q)/q^2\to 0$. Thus
$$\Sum_{q=0}^\infty{|\gamma(q)||z|^q}\leq \Sum_{q=0}^\infty\left(\frac{Cq^2}{p^q}+\frac{|f(q)|}{p^q}\right)\leq (C+1)\Sum_{q=1}^\infty\frac{q^2}{p^q}.$$ It follows that the series $\sum_{q=0}^\infty{\gamma(q)z^q}$ is absolutely convergent. We thus may form the Cauchy product of the series $\sum_{q=0}^\infty{\gamma(q)z^q}$ and $\sum_{k=0}^\infty{z^{j{p^k}}}$ as follows (see \cite{Rudin}, Theorem 3.50):
$$\phi(z)=\Sum_{j=1}^mc_j\Sum_{k=1}^\infty\Sum_{i=0}^k{\gamma(k-i)b_{i,j}z^k},$$ where $$b_{i,j}=\begin{cases}1\text{ if } i=jp^t \text{ for some integer } t\geq 0, \\ 0\text{ otherwise.}\end{cases}$$

Write $S = \{t\in \Z : t\geq 0, jp^t\leq k \}$. It follows that $$\Sum_{k=1}^\infty{\alpha_kz^k}=\Sum_{k=1}^{m-1}\left(\Sum_{j=1}^m\Sum_{i=0}^k{c_j\gamma(k-i)b_{i,j}}\right)z^k+\Sum_{k=m}^\infty\left(\Sum_{j=1}^m\Sum_{t\in S}{c_j\gamma(k-i)b_{i,j}}\right)z^k.$$
Thus for $k\geq m$, $\alpha_k=\Sum_{j=1}^m\Sum_{t\in S}c_j\gamma({k-jp^t})=\Sum_{j=1}^m\Sum_{t=0}^{\left\lfloor\log_p \frac{k}{j}\right\rfloor}c_j\gamma({k-jp^t})$ (if $k\geq m$, then $S \neq \emptyset$). This completes the proof.
\end{proof}
\begin{remark}\label{rem1}
Suppose $n$ is a positive integer. We define $r_n(z)=(1-z^n)(1-z)^2$ for all complex $z$, and consider an automaton whose line complexity sequence satisfies the recursions in Theorem \ref{main}. In this case, we may take $p = 2$, $\phi_T(z)=f_T(z)/z^{n+1}$, and $R_T(z)=\frac{r_n(z)}{z^{n+1}}\left(P_T(z)+\frac{C_\cap z^{2N}}{1-z}\right)$.
\end{remark}
\begin{prop}
Write $\frac{1}{r(z)}=\frac{1}{r_n(z)}=\sum_{k=0}^\infty\eta(k)z^k$. Then $$\eta(k)=\left(1+\left\lfloor\frac{k}{n}\right\rfloor\right)\left(k+1-\frac{n}{2}\left\lfloor\frac{k}{n}\right\rfloor\right).$$
\end{prop}
\begin{proof}We observe that $$\frac{1}{r(z)}=\frac{1}{1-z^n}\frac{1}{(1-z)^2}=\Sum_{k=0}^\infty{\beta_kz^k}\Sum_{q=0}^\infty(q+1)z^q,$$ where $$\beta_k=\begin{cases}1\text{ if } k=nv \text{ for some integer } v\geq 0, \\ 0\text{ otherwise.}\end{cases}$$

We note that the first series is dominated by the geometric series and is hence absolutely convergent. We again form the Cauchy product, obtaining $$\frac{1}{r(z)}=\Sum_{k=0}^\infty\Sum_{q=0}^k{(k-q+1)\beta_qz^k}=\Sum_{k=0}^\infty\Sum_{v=0}^{\lfloor{k/n}\rfloor}(k-nv+1)z^k.$$ Thus, $\eta(k)=\sum_{v=0}^{\lfloor{k/n}\rfloor}(k-nv+1)=(1+\lfloor{k/n}\rfloor)(k+1-\frac{n}{2}\lfloor{k/n}\rfloor)$. This completes the proof.
\end{proof}

\begin{remark}
If $\phi$ is the generating function for a cellular automaton with line complexity $a_R(k)$, then $a_R(k+M)=\alpha(k)\equiv \alpha_k$ for some $M$. We may thus consider the asymptotic behavior of $\alpha(k)/k^2$ to determine that of $a_R(k)/k^2$. For example, if $T$ is as in Remark \ref{rem1}, then $a_T(k+n+1)=\alpha(k)$.
\end{remark}
\begin{remark}
Write $\delta(k)=k\left\lfloor\frac{k}{n}\right\rfloor-\frac{n}{2}\left\lfloor\frac{k}{n}\right\rfloor^2-\frac{k^2}{2n}$. Observe that $\delta(k+n)=\delta(k)$, and if $0\leq k\leq n$, then $-\frac{n}{2}\leq \delta(k)\leq 0$. Thus $\delta(k)=O(1)$. Since $\eta(k)=\frac{k^2}{2n}+\left(k+1-\frac{n}{2}\left\lfloor\frac{k}{n}\right\rfloor+\left\lfloor\frac{k}{n}\right\rfloor\right)+\delta(k)$, we have
$$\eta(k)=\frac{k^2}{2n}+O(k).$$ 
\end{remark}
We now turn to the main result regarding the asymptotic behavior of $\alpha(k)/k^2$. For $y\in\mathbb R$, we will denote the \textit{fractional part} of $y$ by $\langle y\rangle=y-\lfloor y\rfloor$.

\begin{theorem}\label{main_asymp} There exists a continuous, piecewise quadratic function $f$ on $[1/p, 1]$ such that \[\lim_{k\to\infty}\left\lbrack\frac{\alpha(k)}{k^2} - f(p^{-\langle\log_p k\rangle})\right\rbrack = 0.\] The function $f$ is given explicitly by \[f(x) = C\Sum_{j=1}^{\deg R}c_j\left(\frac{p^{2+2\langle\log_p j\rangle - 2\epsilon_j(x)}}{p^2 - 1}x^2 + \frac{2p^{1+\langle\log_p j\rangle - \epsilon_j(x)}}{1-p}x - \lfloor\log_p j\rfloor - \epsilon_j(x)\right),\] where \[\epsilon_j(x) = \begin{cases}1 & \text{ if } \log_p \frac{1}{x} < \langle\log_p j\rangle\\ 0 & \text{ otherwise.}\end{cases}\]
\end{theorem}

\begin{proof}[Proof that the limit is 0]
By Theorem~\ref{coeffrep}, for sufficiently large $k$ we have \begin{align*}
\frac{\alpha(k)}{k^2}&=\Sum_{j=1}^mc_j\Sum_{t=0}^{\flL}\left(\frac{C(k-jp^t)^2}{k^2}+\frac{f(k-jp^t)}{k^2}\right)\\
&=\Sum_{j=1}^mc_j\Sum_{t=0}^{\flL}\frac{C(k-jp^t)^2}{k^2}+O\left(\frac{f(k)\log_p k}{k^2}\right).
\end{align*}

\noindent Thus, $$\lim_{k\to\infty}\left\lbrack\frac{\alpha(k)}{k^2}-\Sum_{j=1}^mc_j\Sum_{t=0}^{\flL}\frac{C(k-jp^t)^2}{k^2}\right\rbrack=\lim_{k\to\infty}\left\lbrack\frac{\alpha(k)}{k^2} - \Sum_{j=1}^mc_j\Sum_{t=0}^{\flL}\left(1-\frac{2jp^t}{k}+\frac{j^2p^{2t}}{k^2}\right)\right\rbrack$$$$=\lim_{k\to\infty}\left\lbrack\frac{\alpha(k)}{k^2}-\left(C\Sum_{j=1}^mc_j\left(\flL+1\right)-C\Sum_{j=1}^mc_j\Sum_{t=0}^{\flL}\frac{2jp^t}{k}+C\Sum_{j=1}^mc_j\Sum_{t=0}^{\flL}\frac{j^2p^{2t}}{k^2}\right)\right\rbrack,$$ where the first expression is equal to \[\lim_{k\to\infty}O\left(\frac{f(k)\log_p k}{k^2}\right)=0.\] 

We will now analyze the three summation terms above. We will write $x(k) = p^{-\langle\log_p k\rangle}$.

\noindent We note that $$C\Sum_{j=1}^mc_j\left(\flL+1\right)=C\Sum_{j=1}^mc_j\left(1-\frL+\log_p k-\log_p j\right).$$ Since $R(1)=0$, we have $\sum_{j=1}^mc_j=0$. The above expression reduces to $$-C\Sum_{j=1}^mc_j\frL-C\Sum_{j=1}^mc_j\log_p j.$$ We have
 $$\frL=\begin{cases}\left\langle\log_p k\right\rangle-\left\langle\log_p j\right\rangle & \text{ if } \left\langle\log_p k\right\rangle\geq \left\langle\log_p j\right\rangle \\ \left\langle\log_p k\right\rangle-\left\langle\log_p j\right\rangle+1 & \text{ otherwise,}\end{cases}$$ so that $\frL = \left\langle\log_p k\right\rangle - \left\langle\log_p j\right\rangle+\epsilon_j(x(k))$.  
\noindent Therefore, \begin{align*}
C\Sum_{j=1}^mc_j\left(\flL+1\right)&=-C\Sum_{j=1}^mc_j\left\langle\log_p k\right\rangle+C\Sum_{j=1}^mc_j\left(\left\langle\log_p j\right\rangle-\epsilon_j(x(k))-\log_pj\right)\\
&=-C\Sum_{j=1}^mc_j\log_p\frac{1}{x(k)}+C\Sum_{j=1}^mc_j\left(\left\langle\log_p j\right\rangle-\log_p j-\epsilon_j(x(k))\right)\\
&=-C\Sum_{j=1}^mc_j\left(\left\lfloor\log_p j\right\rfloor+\epsilon_j(x(k))\right).
\end{align*}

\noindent We have that \begin{align*}-\Sum_{t=0}^{\flL}\frac{2jp^t}{k}&=-2\Sum_{t=0}^{\flL}\frac{1}{p^{\log_p\frac{k}{j}}}p^t\\&=\frac{2}{p^{\log_p{\frac{k}{j}}}}\frac{p^{\flL+1}-1}{1-p}\\&=\frac{2}{1-p}\left(p^{-\frL+1}-p^{-\rL}\right)\\&=\frac{2}{1-p}\left(p^{-\left\langle\log_p k\right\rangle}p^{\left\langle\log_p j\right\rangle}p^{1-\epsilon_j(x(k))}-\frac{j}{k}\right).\end{align*} Thus, $$-C\Sum_{j=1}^mc_j\Sum_{t=0}^{\flL}\frac{2jp^t}{k}=C\Sum_{j=1}^m\frac{2c_j}{1-p}p^{-\log_p\frac{1}{x(k)}}p^{1+\left\langle\log_p j\right\rangle-\epsilon_j(x(k))}=C\Sum_{j=1}^m\frac{2c_jp^{1+\langle\log_p j\rangle-\epsilon_j(x(k))}}{1-p}x(k).$$ Similarly, \begin{align*}
\Sum_{t=0}^{\flL}\left(\frac{j}{k}\right)^2\left(p^2\right)^t&=\frac{1}{p^{2\log_p\frac{k}{j}}}\frac{1-p^{2\flL+2}}{1-p^2}\\
&=\frac{1}{1-p^2}\left(p^{-2\log_p\frac{k}{j}}-p^{2-2\frL}\right)\\
&=\frac{1}{1-p^2}\left(\frac{j^2}{k^2}-p^{2-2\left\langle\log_p k\right\rangle+2\left\langle\log_p j\right\rangle-2\epsilon_j(x(k))}\right).
\end{align*}
Thus, \begin{align*}
C\Sum_{j=1}^mc_j\Sum_{t=0}^{\flL}\left(\frac{j}{k}\right)^2p^{2t}&=-C\Sum_{j=1}^m\frac{c_j}{1-p^2}p^{2\log_p x(k)}p^{2+2\left\langle\log_p j\right\rangle-2\epsilon_j(x(k))}\\
&=C\Sum_{j=1}^m\frac{c_jp^{2+2\langle\log_p j\rangle - 2\epsilon_j(x(k))}}{p^2-1}x(k)^2.
\end{align*}
It follows that \[\lim_{k\to\infty}\left\lbrack\frac{\alpha(k)}{k^2} - f(x(k))\right\rbrack = 0.\]

%Lemma formerly used in the proof that $f$ is continuous:

%\begin{lemma}
%Suppose $j_0 = m_0p^{k_0}$, where $k_0$, $m_0$ are integers with $k_0\geq 0$, $m_0>0$, and suppose that either $m_0=1$ or $(m_0, p)=1$. Then if $\langle\log_p j_0\rangle = \langle\log_p j\rangle$, we have \[j = m_0p^k\] for some integer $k$.
%\end{lemma}
%\begin{proof}[Proof of the Lemma]
%Write $j = mp^k$, for some integer $m$ with $m = 1$ or $(m, p) = 1$. We will show that $m = m_0$. Since $\langle\log_p j_0\rangle = \langle\log_p j\rangle$, we have $\langle\log_p m_0\rangle = \langle\log_p m\rangle$. Let $r, r_0$ be the largest integers with $p^r \leq m$ and $p^{r_0}\leq m_0$, respectively. It follows that \[\langle\log_p m\rangle = \log_p m - r,\] and similarly for $r_0, m_0$, so that \[\log_p \frac{m_0}{m} = r_0 - r.\] Assume that $m_0\geq m$; otherwise we can switch $m, m_0$ and $r, r_0$. Since $m_0\geq m$, we clearly have $r_0 \geq r$. If $r_0 > r$, it follows that $p ~|~ m_0$, a contradiction. Thus $r_0 = r$, so that $m_0 = m$. The result follows.
\end{proof}

\begin{proof}[Proof that $f$ is continuous] Let $D = \{p^{-\langle\log_p j\rangle}:1\leq j \leq \deg R\}$. It is clear that $f$ is continuous at $x$ when $x\notin D$. Moreover, since each $\epsilon_j$ is left-continuous, it follows that $f$ is left-continuous everywhere on $[1/p, 1]$. Now, fix an arbitrary $j_0$ with $1\leq j_0\leq \deg R$, and let $x_0 = p^{-\langle\log_p j_0\rangle}$. We will show that \[\lim_{x\to x_0^+}f(x) = f(x_0),\] so that $f$ is continuous at $x_0$.

We first observe that the function $f$ is a finite sum of $\deg R$ functions, where the $j$th term of the sum is continuous everywhere except possibly at $x = p^{-\langle\log_p j\rangle}$; in particular, if $j$ is such that $x_0 \neq p^{-\langle\log_p j\rangle}$, then the $j$th term in the sum is continuous at $x_0$. %Moreover, if we write $j_0 = m_0p^{k_0}$ as above, the lemma shows that any $j$ with $x_0 = p^{-\langle\log_p j\rangle}$ must be of the form $j = m_0p^k$ for some $k$. 

\noindent Define \[J(x_0) = \{j:x_0 = p^{-\langle\log_p j\rangle}\}.\]% = \{m_0p^k ~|~ 1\leq m_0p^k \leq\deg R\}.\] 
\noindent The above observations show that \begin{align*}\lim_{x\to x_0^+}f(x) - f(x_0) =~& C\Sum_{j\in J(x_0)}c_j\left(\frac{p^{2+2\langle\log_p j\rangle - 2}}{p^2 - 1}p^{-2\langle\log_p j_0\rangle} + \frac{2p^{1+\langle\log_p j\rangle-1}}{1-p}p^{-\langle\log_p j_0\rangle}-\lfloor\log_p j\rfloor\right) \\&-C\Sum_{j\in J(x_0)}c_j \\&-C\Sum_{j\in J(x_0)}c_j\left(\frac{p^{2+2\langle\log_p j\rangle}}{p^2 - 1}p^{-2\langle\log_p j_0\rangle} + \frac{2p^{1+\langle\log_p j\rangle}}{1-p}p^{-\langle\log_p j_0\rangle}-\lfloor\log_p j\rfloor\right).\end{align*} We note that $\langle\log_p j\rangle = \langle\log_p j_0\rangle$ for all $j$ in $J(x_0)$, so that the above expression reduces to \[C\Sum_{j\in J(x_0)}c_j\left(\frac{1}{p^2-1}+\frac{2}{1-p}-\frac{p^2}{p^2-1}-\frac{2p}{1-p}\right)-C\Sum_{j\in J(x_0)}c_j = 0.\] Thus $f$ is continuous at $x_0$. Since $j_0$ was arbitrary, it follows that $f$ is continuous on $[1/p, 1]$. This completes the proof of Theorem \ref{main_asymp}.

\end{proof}

This theorem has several interesting consequences. We can use it to investigate the behavior of specific subsequences of the quotient $\frac{\alpha(k)}{k^2}$:

\begin{cor}\label{subseq_asymp}
For $\frac{1}{p}\leq x\leq 1$, let $s_k(x)$ be a sequence of positive integers such that $s_k\to\infty$ and $\lim_{k\to\infty}\langle\log_p s_k(x)\rangle = \log_p \frac{1}{x}$. Then \[\lim_{k\to\infty}\frac{\alpha(s_k(x))}{s_k(x)^2} = f(x).\]
\end{cor}

\begin{proof}
With $x(k)$ as in Theorem \ref{main_asymp}, the hypothesis ensures that $x(k)\to x$. The result thus follows by the continuity of $f$.
\end{proof}

The next corollary shows that the limit superior and limit inferior of $\frac{\alpha(k)}{k^2}$ can be determined explicitly using the function $f$.

\begin{cor}
We have \[\liminf_{k\to\infty}\frac{\alpha(k)}{k^2} = \inf_{\frac{1}{p}\leq x\leq 1} f(x)\] and \[\limsup_{k\to\infty}\frac{\alpha(k)}{k^2} = \sup_{\frac{1}{p}\leq x\leq 1}f(x).\]
\end{cor}

\begin{proof}
We prove the second statement (the first follows in the same way). Fix $\epsilon>0$. By Theorem \ref{main_asymp}, there exists a positive integer $n$ such that \[\frac{\alpha(k)}{k^2} < f(p^{-\langle\log_p k\rangle}) + \epsilon\] for all $k\geq n$, so that \[\sup_{k\geq n}\frac{\alpha(k)}{k^2}\leq \sup_{\frac{1}{p}\leq x\leq 1}f(x)+\epsilon.\] It follows that \[\sup_{\frac{1}{p}\leq x\leq 1}f(x) = \inf_{n\geq 1}\sup_{k\geq n}\frac{\alpha(k)}{k^2} = \limsup_{k\to\infty}\frac{\alpha(k)}{k^2}.\]
\end{proof}

In the next corollary, we return to the case where $p = 2$ and $T$ is not suspicious.

\begin{cor}
Suppose $T$ is not suspicious, $\frac{1}{2}\leq x\leq 1$, $s_k$ is as in Corollary \ref{subseq_asymp}, and $R_n(z)=\sum_{j=1}^{\deg R_n}v_jz^j$. Then $$\lim_{k\to\infty}\frac{a_T(s_k)}{s_k^2}=\frac{1}{2n}\Sum_{j=1}^{\deg R_n}v_j\left(\frac{2^{2 + 2\langle\log_2 j\rangle - 2\epsilon_j(x)}}{3}x^2-2^{2 + \langle\log_2 j\rangle - \epsilon_j(x)}x-\lfloor\log_2 j\rfloor-\epsilon_j(x)\right).$$
\end{cor}

\noindent\textbf{Example 1.}
Let $T=1+x+x^3$. Then if we denote the limit function above by $f_3(x)$, explicit calculation using a computer yields

$$f_3(x)=\begin{cases}-\mathlarger{\frac{15}{32}}x^2 + \mathlarger{\frac{7}{12}}x + \mathlarger{\frac{11}{6}} &\text{ if } \mathlarger{\frac{1}{2}}\leq x<\mathlarger{\frac{2}{3}}\\\\-\mathlarger{\frac{3}{32}}x^2 + \mathlarger{\frac{1}{12}}x + 2 &\text{ if } \mathlarger{\frac{2}{3}}\leq x< \mathlarger{\frac{4}{5}}\\\\\mathlarger{\frac{41}{96}}x^2 - \mathlarger{\frac{3}{4}}x + \mathlarger{\frac{7}{3}} &\text{ if } \mathlarger{\frac{4}{5}}\leq x<\mathlarger{\frac{8}{9}}\\\\ \mathlarger{\frac{83}{384}}x^2 - \mathlarger{\frac{3}{8}}x + \mathlarger{\frac{13}{6}} &\text{ if } \mathlarger{\frac{8}{9}}\leq x\leq 1.\end{cases}$$

\noindent We note that the maximum and minimum of $f_3$ are $\frac{272}{135}$ and $\frac{493}{246}$, respectively. Thus, $$\limsup_{k\to\infty}\frac{a_T(k)}{k^2}=\frac{272}{135}$$ and $$\liminf_{k\to\infty}\frac{a_T(k)}{k^2}=\frac{493}{246}.$$

See Figure 5 for a graph of $f_3$. See Figure 6 for an illustration of the convergence of $a_T(s_k)/s_k^2$ to $f_3$.
\begin{figure}
\begin{center}
\includegraphics[scale=0.7]{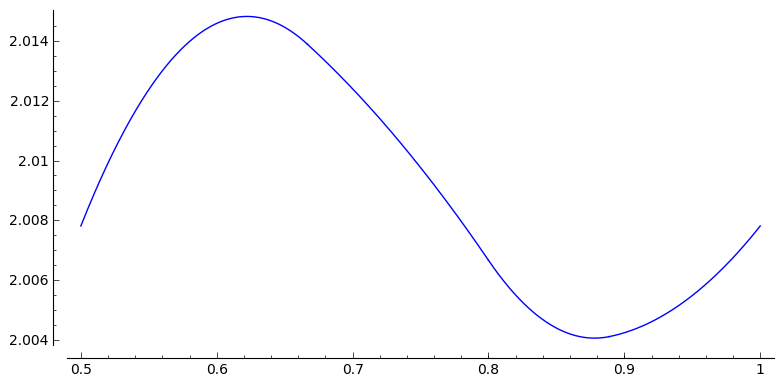}
\caption{Plot of the limit function for the case $n=3$}
\end{center}
\end{figure}
\begin{figure}
\begin{center}
\includegraphics[scale=0.7]{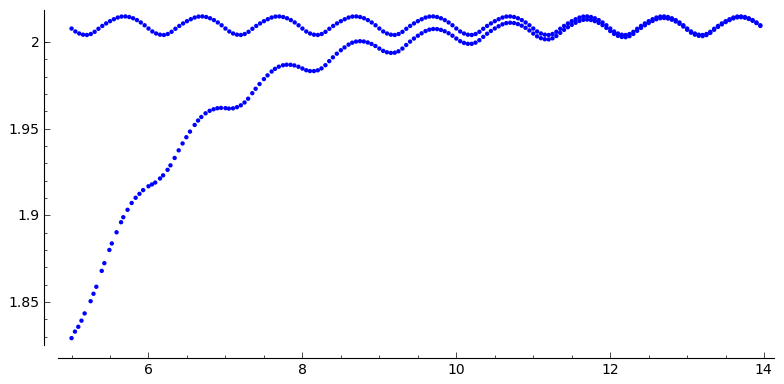}
\end{center}
\caption{We plot $f_3(2^{-\langle \log_2 y\rangle})$ (above) and $a_T(\lfloor y\rfloor)/{\lfloor y\rfloor}^2$ (below) versus $\log_2 y$ (horizontal), where for each function we sample various points $y\in[2^5,2^{14}]$.}
\end{figure}
\hspace{1mm}\\\hspace{1mm}\\\hspace{1mm}
\noindent\textbf{Example 2.}
Let $T=1+x+x^4$. Then, denoting the limit function by $f_4(x)$, we obtain the following representation (again using a computer)

$$f_4(x)=\begin{cases}-\mathlarger{\frac{235}{192}}x^2 + \mathlarger{\frac{11}{8}}x + \mathlarger{\frac{15}{8}} &\text{ if } \mathlarger{\frac{1}{2}}\leq x<\mathlarger{\frac{8}{15}}\\\\-\mathlarger{\frac{1205}{1536}}x^2 + \mathlarger{\frac{29}{32}}x + 2 &\text{ if } \mathlarger{\frac{8}{15}}\leq x<\mathlarger{\frac{4}{7}}\\\\-\mathlarger{\frac{617}{1536}}x^2 + \mathlarger{\frac{15}{32}}x + \mathlarger{\frac{17}{8}} &\text{ if } \mathlarger{\frac{4}{7}}\leq x<\mathlarger{\frac{8}{13}}\\\\-\mathlarger{\frac{55}{768}}x^2 + \mathlarger{\frac{1}{16}}x + \mathlarger{\frac{9}{4}} &\text{ if } \mathlarger{\frac{8}{13}}\leq x<\mathlarger{\frac{4}{5}}\\\\\mathlarger{\frac{245}{768}}x^2 - \mathlarger{\frac{9}{16}}x + \mathlarger{\frac{5}{2}} &\text{ if } \mathlarger{\frac{4}{5}}\leq x\leq1.\end{cases}$$
\noindent We note that the maximum and minimum of $f_4$ are given by $\frac{2791}{1234}$ and $\frac{2207}{980}$, respectively. Thus, $$\limsup_{k\to\infty}\frac{a_T(k)}{k^2}=\frac{2791}{1234}$$ and $$\liminf_{k\to\infty}\frac{a_T(k)}{k^2}=\frac{2207}{980}.$$

See Figure 7 for an illustration of the convergence of $a_T(s_k)/s_k^2$ to $f_4$.
% manual label override
\begin{figure}
\begin{center}
\includegraphics[scale=0.7]{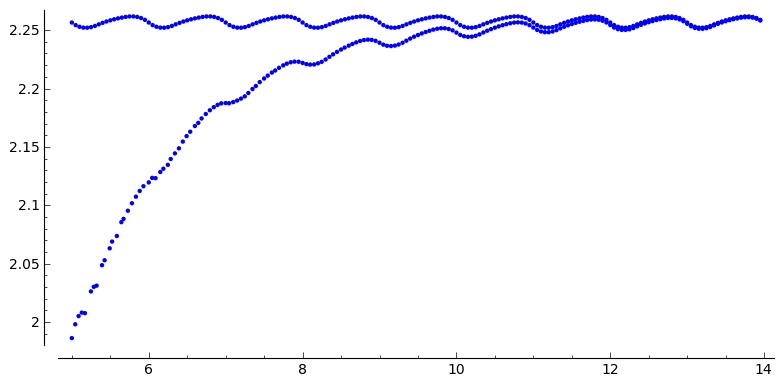}
\caption{$f_4(2^{-\langle \log_2 y\rangle})$ (above) and $a_T(\lfloor y\rfloor)/{\lfloor y\rfloor}^2$ (below) versus $\log_2 y$ (horizontal)}
\end{center}
\end{figure}

\section{Conclusion}

We have investigated recursion formulas for the line complexity sequence where the number of accessible blocks of length $2k$ is expressed in terms of the numbers of accessible blocks of several different smaller lengths. These recursions are intimately connected with the sets $A_i$, $B_i$ and the maps $T_{A_i}$, $T_{B_i}$ introduced above; in particular, we require these maps to be injective. The maps $T_{A_i}$ are always injective, but the same need not be true of the maps $T_{B_i}$. By closely analyzing the maps $T_{B_i}$, we have precisely characterized the polynomials that are not suspicious, i.e. those for which the maps $T_{B_i}$ are injective on the whole space. We have also proved that for many polynomial transition rules $T$, the intersections of at least two of the sets $A_i$ and $B_i$ are of constant size for sufficiently large $k$. This result leads to the recursions in Theorem \ref{main}, which hold for all positive-degree polynomials that are not suspicious.

We have also investigated the behavior of the line complexity sequence when the transition rule is raised to different powers; by introducing a notion of the \emph{order} of a recursion distinct from the order of the transition rule, we have seen that if an automaton modulo $p$ with a constant initial state and a transition rule $T$ satisfies a recursion of some order $r$, the automaton whose transition rule is $T^n$ (for any $n$) satisfies a recursion of order $rp^s$, where $s$ is the largest integer such that $p^s \mid n$.

In addition, we have proved functional relations for the generating functions associated to the sequence $a_T(k)$ in the mod $2$ case. In a more general setting, we proved that if $\phi(z)=\sum_{k=1}^\infty\alpha(k)z^k$ satisfies a certain functional equation relating $\phi(z)$ and $\phi(z^p)$, there is a continuous, piecewise quadratic function $f$ on $[1/p, 1]$ for which $\lim_{k\to\infty}\left\lbrack\frac{\alpha(k)}{k^2} - f(p^{-\langle\log_p k\rangle})\right\rbrack = 0$. Using this result, we have shown that for positive integer sequences $s_k\to\infty$ with a parameter $x\in [1/p,1]$ and for which $\lim_{k\to\infty}\langle\log_p s_k(x)\rangle=\log_p\frac{1}{x}$, the ratio $\alpha(s_k(x))/s_k(x)^2$ tends to $f(x)$. We have also shown that the limit superior and inferior of $\alpha(k)/k^2$ are given explicitly by the extremal values of $f$.

The requirement that a polynomial be nonsuspicious seems to be a very natural condition for recursions of the above type to exist; but in fact, numerical evidence suggests that some suspicious polynomials may satisfy such recursions. For example, direct computation of the line complexity sequence suggests that the rule $T=1+x+x^3+x^4$ satisfies a recursion of order $4$. We have also seen that powers of polynomials may satisfy lower-order recursions, as in the last section. 

It is of note that for recursions like those in Theorem \ref{main} to hold, we only require the maps $T_{B_i}$, $T_{B_i}'$ to be injective on the sets of accessible blocks on which they are defined, not necessarily on the whole space. The example of $T=1+x+x^3+x^4$ is interesting, in that the maps appear to be very nearly injective in this sense: explicit computation of the image ($k = 20, 30,$ etc.) suggests that the only blocks in the range for which injectivity fails are of the form $1010\cdots$ and $100100\cdots$, including translates. In particular, it appears that there are always four of them. We thus conjecture that this $T$ satisfies the recursions of Theorem \ref{main}, and that other polynomials might exhibit similar behavior.

On the other hand, some polynomials do not seem to satisfy a recursion of any order; one example is $T=1+x^2+x^3+x^5$. These observations, in connection with the observation that irreducible polynomials are not suspicious, suggest that recursive behavior of the line complexity sequence may be related to factorization properties of the polynomial; this constitutes perhaps the most immediate direction of further research.

\newpage

Other research directions include considering automata with coefficients taken modulo $p$, to see if the behaviors that arise in these situations are analogous to those we have observed in the present case.

\section{Acknowledgments}

I would like to thank Mr. Chiheon Kim for mentoring this project and for providing many helpful insights and suggestions. I would like to thank Prof. Pavel Etingof for suggesting this project, and Prof. Richard Stanley for suggesting the original topic. I would also like to thank the Center for Excellence in Education, the Massachusetts Institute of Technology, and the MIT Math Department for making RSI possible. I would like to thank Mr. Antoni Rangachev and Dr. Tanya Khovanova for their advice regarding this paper. I would also like to thank my sponsors, Mr. Steven Ferrucci of the American Mathematical Society, Dr. Donald McClure of the American Mathematical Society, Mr. and Mrs. Raymond C. Kubacki, Mr. Piotr Mitros, Mr. and Mrs. Steven Scott, Prof. Tom Leighton of Akamai Technologies, and Prof. Bonnie Berger of MIT. Finally, I would like to thank the MIT Math Department, the UROP+ program, and the Class of 1994 UROP fund for making possible the continuation of this project in the summer of 2015.

\newpage

\bibliographystyle{style}

\bibliography{bibliog}
\end{document}